\documentclass[12pt,reqno]{amsart}
\usepackage{amsfonts,amsmath,amscd,amssymb,amsthm}
\usepackage{enumerate,hyperref,perpage,url}
\usepackage[margin=2.2cm, a4paper]{geometry}

\theoremstyle{plain}
\newtheorem{thm}{Theorem}[section]

\newtheorem{cor}[thm]{Corollary}
\newtheorem{lemma}[thm]{Lemma}
\newtheorem{prop}[thm]{Proposition}
\newtheorem{question}[thm]{Question}
\theoremstyle{remark}
\newtheorem*{defn}{\textbf{Definition}}
\newtheorem*{ex}{\textbf{Example}}
\newtheorem*{rmk}{\textbf{Remark}}

\numberwithin{equation}{section}

\MakePerPage{footnote} \allowdisplaybreaks 

\newcommand{\Av}{\mathrm{Av}}

\newcommand{\comp}{\mathrm{comp}}

\newcommand{\Exp}{\mathrm{Exp}}

\newcommand\M{\mathbb M}
\newcommand\N{\mathbb N}
\newcommand{\ol}{\overline}
\newcommand{\Op}{\mathrm{Op}}
\newcommand\R{\mathbb R}

\newcommand{\supp}{\mathrm{supp}\,}

\newcommand{\ve}{\varepsilon}

\newcommand{\Vol}{\mathrm{Vol}}
\newcommand{\w}{\omega}

\title[Small scale quantum ergodicity]{Small scale quantum ergodicity in negatively curved manifolds}

\author{Xiaolong Han}
\address{Department of Mathematics, The Australian National University, Canberra, ACT 2601, Australia}
\email{Xiaolong.Han@anu.edu.au}



\thanks{Research is partially supported by the Australian Research Council through Discovery Project DP120102019.}

\date{}

\begin{document}
\maketitle

\begin{abstract}
In this paper, we investigate quantum ergodicity in negatively curved manifolds. We consider the symbols depending on a semiclassical parameter $h$ with support shrinking down to a point as $h\to0$. The rate of shrinking is a power of $|\log h|$. This extends the asymptotic equidistribution of quantum ergodic eigenfunctions to a logarithmical scale. 
\end{abstract}

\section{Introduction}
Quantum ergodicity studies the quantized counterpart of a classical dynamical system that is ergodic. In this paper, we consider the geodesic flow in a manifold. Let $(\M,g)$ be a compact and smooth Riemannian manifold of dimension $n$ without boundary. Denote a state as $(x,\xi)$ in the cotangent bundle $T^*\M$. Let $H(x,\xi)=|\xi|^2_x$ be the Hamiltonian, where $|\cdot|_x$ is the induced metric in $T_x^*\M$. Then the geodesic flow $G_t$ is generated by $X_H$, the Hamiltonian vector field of $H$. It preserves the canonical symplectic form $\w=d\xi\wedge dx$, hence preserves the Liouville volume form $\mu=\w^n/n!$. 

Let $S^*\M=\{(x,\xi)\in T^*\M:|\xi|_x=1\}$, i.e. the energy layer at $H(x,\xi)=1$. Then $G_t$ preserves $S^*\M$, and induces an invariant Liouville measure $\mu_L$ on $S^*\M$. We say that $G_t$ is ergodic on $S^*\M$ with respect to $\mu_L$ if any invariant subset of $S^*\M$ under $G_t$ has $\mu_L$-measure $0$ or $\mu_L(S^*\M)$. We normalize the Liouville measure and let $\mu_1=\mu_L/\mu_L(S^*\M)$.

The quantum system of $(S^*\M,G_t)$ involves the eigenfunctions of the (positive) Laplace-Beltrami operator $\Delta=\Delta_g$: Write $\{u_j\}_{j=0}^\infty$ as an orthonormal basis of eigenfunctions (ONBE) of $\Delta$ with eigenvalues $\lambda_j^2$, i.e. $\Delta u_j=\lambda_j^2u_j$.

The quantum ergodic theorem of \v Snirel'man-Zelditch-Colin de Verdi\`ere \cite{Sn, Ze1, CdV} illustrates the correspondence between the classical system (geodesic flow on $S^*\M$) and the quantum system (ONBE) when $G_t$ is ergodic on $S^*\M$ with respect to the Liouville measure $\mu_1$.

\begin{thm}[Quantum ergodicity, high-energy version]\label{thm:QE} 
Assume that $G_t$ is ergodic on $S^*\M$ with respect to the Liouville measure $\mu_1$. Then
\begin{equation}\label{eq:QE}
S_1(\lambda,a)=\frac{1}{N(\lambda)}\sum_{\lambda_j\le\lambda}\Big|\langle\Op(a)u_j,u_j\big\rangle-\mu_1(a)\Big|=o(1)\quad\text{as }\lambda\to\infty,
\end{equation}
for any orthonormal basis of eigenfunctions $\{u_j\}_{j=0}^\infty$.
\end{thm}

We call it the high-energy version since it is on the high-energy limit $\lambda^2\to\infty$. Here,  
\begin{itemize}
\item $N(\lambda)=\#\{j:\lambda_j\le\lambda\}$ is the eigenfrequency counting function; Weyl's law yields that  
$$N(\lambda)=\frac{\mu(B^*\M)}{(2\pi)^n}\lambda^n+O(\lambda^{n-1}),$$
where $B^*\M=\{(x,\xi)\in T^*\M:|\xi|_x\le1\}$ is the coball bundle of $\M$. See e.g. H\"ormander \cite{Ho}.
\item $\langle\cdot,\cdot\rangle$ is the inner product in $L^2(\M)$ with respect to the Riemannian volume $\Vol=\Vol_g$ induced by the metric $g$:
$$\langle f,g\rangle=\int_\M f\,\ol g\,d\Vol\quad\text{for }f,g\in L^2(\M).$$
\item $\Op(a)$ is a classical pseudodifferential operator of order $0$ with principal symbol $a$. 
\item $\mu_1(a)$ is the space-average of $a$ with respect to the Liouville measure $\mu_1$ on $S^*\M$:
$$\mu_1(a)=\int_{S^*\M}a\,d\mu_1.$$
\end{itemize}

If all the sectional curvatures of $\M$ are negative everywhere, then we say that $\M$ is negatively curved. In a negatively curved manifold, it is well known that the geodesic flow $G_t$ is ergodic on $S^*\M$ with respect to the Liouville measure. (See e.g. Katok and Hasselblatt \cite{KH}.) Moreover, $G_t$ is an Anosov flow and displays stronger properties than ergodicity, e.g. central limiting, strong-mixing, and exponentially decay of correlations, etc. (See Section \ref{sec:geo} for more discussion.) In this case, one expects a more qualitative version of Theorem \ref{thm:QE}. Define the $p$-moment
$$S_p(\lambda,a)=\frac{1}{N(\lambda)}\sum_{\lambda_j\le\lambda}\Big|\langle\Op(a)u_j,u_j\big\rangle-\mu_1(a)\Big|^p\quad\text{for }p>0.$$
Zelditch \cite{Ze2} explored Ratner \cite{R}'s central limit theorem of $G_t$ to improve Theorem \ref{thm:QE} in negatively curved manifolds to
\begin{equation}\label{eq:Zelditch}
S_p(\lambda,a)=O\left((\log\lambda)^{-\frac p2}\right).
\end{equation}
See also Schubert \cite{Sc} in a setting with more general Hamiltonian. Sarnak \cite{Sa1} conjectured that on compact hyperbolic surfaces, 
$$S_1(\lambda,a)=O\left(\lambda^{-\frac14+\ve}\right)$$
for any small $\ve>0$. In fact, his conjecture is on the individual terms:
$$\Big|\big\langle\Op(a)u_j,u_j\big\rangle-\mu_1(a)\Big|=O\left(\lambda_j^{-\frac14+\ve}\right),$$
which would imply that 
\begin{equation}\label{eq:QUE}
\big\langle\Op(a)u_j,u_j\big\rangle\to\mu_1(a)\quad\text{as }j\to\infty.
\end{equation}
Quantum unique ergodicity (QUE) conjecture states that \eqref{eq:QUE} is valid on any compact manifold of (possibly variable) negative curvature, see 
Rudnick and Sarnak \cite{RS}. The statement is on the asymptotic behavior of the whole ONBE $\{u_j\}_{j=0}^\infty$. However, it is very difficult to prove such result. Restricting to certain orthonormal bases of eigenfunctions, QUE has been verified in special cases when $(\M,g)$ is arithmetic, by Lindenstrauss \cite{Lin}, Silbermann and Venkatesh \cite{SV}, Holowinsky and Soundararajan \cite{HS}, and Brooks and Lindenstrauss \cite{BL}.

In the case when $G_t$ is only ergodic, Theorem \ref{thm:QE} implies that \eqref{eq:QUE} holds for ``\textit{almost all}'' (that is, full density as defined in the following) eigenfunctions in any ONBE. 
We define the density $D$ of a subsequence $J=\{j_k\}\subset\N$ as
$$D(J)=\lim_{N\to\infty}\frac{\#\{j_k<N\}}{N}\quad\text{if it exists}.$$
When $D=0$ ($>0$ or $=1$), we call such subsequence a zero (positive or full) density subsequence. The full density subsequence is also called density one subsequence. It follows from a very standard argument to extract a full density subsequence of $\{u_j\}$ for which \eqref{eq:QUE} is true, using $S_p(\lambda,a)=o(1)$ for some $p>0$. See e.g. \cite[Lemma 3]{Ze1}.

\begin{cor}\label{cor:QE}
Assume that $G_t$ is ergodic on $S^*\M$ with respect to the Liouville measure $\mu_1$. Given any ONBE $\{u_j\}_{j=0}^\infty$, there exists a full density  subsequence of eigenfunctions $\{u_{j_k}\}\subset\{u_j\}$ such that 
$$\big\langle\Op(a)u_{j_k},u_{j_k}\big\rangle\to\mu_1(a)\quad\text{as }k\to\infty,$$
for each classical pseudodifferential operator $\Op(a)$ of order $0$. Then $\{u_{j_k}\}$ is called a quantum ergodic subsequence of eigenfunctions.
\end{cor}

This corollary hence states that the semiclassical measure defined by the full density subsequence $\{u_{j_k}\}$ in $T^*\M$ coincides with the Liouville measure $\mu_1$ on $S^*\M$. It, however, does not exclude the possibility that there might exist a sparse subsequence of $\{u_j\}$ for which \eqref{eq:QUE} is not valid. In fact, Hassell \cite{Ha} provides an example that supports such a subsequence, that is, the generic Bunimovich stadia. This also verifies that classical ergodicity of $G_t$ can not imply QUE.

Let $\Omega\subset\M$ be a Borel subset with measure-zero boundary. Then by Portmanteau theorem (c.f. \cite[Theorem 6.2.5]{So1}), we have
\begin{equation}\label{eq:equidinM}
\int_\Omega|u_{j_k}|^2\,d\Vol\to\frac{\Vol(\Omega)}{\Vol(\M)}\quad\text{as }k\to\infty.
\end{equation}
Therefore, for this quantum ergodic subsequence of eigenfunctions, the $L^2$ mass of $\{u_{j_k}\}$ display equidistribution in $\M$ as $k\to\infty$. (Theorem \ref{thm:QE} actually concludes asymptotic equidistribution on $S^*\M$.) Here, the set $\Omega$ is fixed and we say that it is of scale $O(1)$. In the classical-quantum correspondence, an eigenfunction $u_j$ represents a stable state of a freely moving particle in $\M$, and the $L^2$ mass distribution of $u_j$ is interpreted as the probability density of finding a particle in $\M$. Therefore, \eqref{eq:equidinM} indicates that for almost all stable states, the probability density tends to the normalized Riemannian volume $d\Vol/\Vol(\M)$.

The $L^2$ distribution characterization of eigenfunctions has a lot of applications beyond its connection with quantum physics. In some applications, one is interested in the mass distribution in regions of smaller scales. For example, let $B(x,r)$ is the geodesic ball centered at $x\in\M$ with radius $r$. Consider the eigenfunction $u_j$ in $B(x,\lambda_j^{-1})$. One can rescale $B(x,\lambda_j^{-1})$ to the ball with radius $1$. Then the operator $\Delta-\lambda_j^2$ is a small perturbation of the Euclidean Laplacian in the unit ball. Hence, the classical elliptic PDE theory applies, e.g. the mean value inequality (c.f. Li and Schoen \cite{LS}):
$$\|u_j\|_{L^\infty(B(x,\lambda_j/2)}^2\le C\lambda_j^{-n}\int_{B(x,\lambda_j)}|u_j|^2\,d\Vol.$$ 
So the $L^2$ mass of $u_j$ in $B(x,\lambda_j)$ controls its supreme value in the half ball. This approach has been used in the nodal set estimates of eigenfunctions, see e.g. Donnelly and Fefferman \cite{DF}, Colding and Minicozzi \cite{CM}. However, in such a scale $\lambda_j^{-1}$ there is no fast oscillation, therefore generally no results on mass equidistribution can be expected. Hence, we ask the following question.

\begin{question}[Small scale mass equidistribution]\label{q:deltaequidinM}
Let $\rho\in(0,1)$. Given any ONBE, does there exist a full density subsequence $\{u_{j_k}\}$ such that
\begin{equation}\label{eq:rhoequidinM}
\int_{B(x,r_{j_k})}|u_{j_k}|^2\,d\Vol=\frac{\Vol(B(x,r_{j_k}))}{\Vol(\M)}+o(r_{j_k}^n)\quad\text{as }k\to\infty
\end{equation}
for $r_{j_k}=r(\lambda_{j_k})=\lambda_{j_k}^{-\rho}$ and all $x\in\M$?
\end{question} 

This statement, if true, asserts that the full density subsequence $\{u_{j_k}\}$ is $L^2$ mass equidistributed in the scale $\lambda_{j_k}^{-\rho}$. For example, on the torus $\mathbb T^n$, the whole ONBE $\{e^{in\cdot x}\}$ is equidistributed in any small scale; for small scale mass equidistribution of an arbitrary ONBE on $\mathbb T^n$, see Hezari and Rivi\`ere \cite{HR2} and the reference therein.

In negatively curved manifolds, we prove small scale mass equidistribution when the scale $r(\lambda_{j_k})=(\log\lambda_{j_k})^{-\alpha}$ for some $\alpha>0$: Given any ONBE, for a fixed point $x_0\in\M$, we show that \eqref{eq:rhoequidinM} is valid for a full density subsequence (depending on $x_0$) in the scale $(\log\lambda_{j_k})^{-\alpha}$ for $\alpha<1/(2n)$ (see Corollary \ref{cor:deltaequidatx}); for all the points on $\M$, we show that the two sides of \eqref{eq:rhoequidinM} are uniformly comparable for a full density subsequence in the scale $(\log\lambda_{j_k})^{-\alpha}$ for $\alpha<1/(3n)$ (see Corollary \ref{cor:deltaequidinM}).

\begin{rmk}\hfill
\begin{enumerate}
\item In negatively curved manifolds at a given point $x$, \eqref{eq:rhoequidinM} is automatically true if $\rho=0$ by \eqref{eq:equidinM}. Here, we assume that the injectivity radius of $\M$ is much greater than $1$ (and throughout the whole paper).
\item We focus on the negatively curved manifolds since more results on the dynamical properties of the geodesic flows are available in this case, particularly, Liverani \cite{Liv}'s exponential decay of correlation. It is doubtful that \eqref{eq:equidinM} can be improved to smaller scales in a manifold with only ergodicity assumption on the geodesic flow.
\item It is a natural question to ask whether the convergence in \eqref{eq:rhoequidinM} is uniform for $x\in\M$. Such uniform convergence is crucial for its application, because one can use equidistribution to study the eigenfunctions locally in each small region and then combine them to a global estimate in the whole manifold by uniformity. 
\end{enumerate}
\end{rmk}

It is convenient to work in the semiclassical setting. In this setting, the quantum ergodicity theorem is proved by Helffer, Martinez, and Robert \cite{HMR}. Here, we state the result from \cite[Theorem 5]{DyGu}.
 
\begin{thm}[Quantum ergodicity, a semiclassical version]\label{thm:scQE}
Assume that $G_t$ is ergodic on $S^*\M$ with respect to the Liouville measure $\mu_1$. Let $\{u_j\}_{j=0}^\infty$ be an ONBE of the quantized Laplacian, $h^2\Delta$, that is, $h^2\Delta u_j=E_ju_j$. Then
\begin{equation}\label{eq:scQE}
V_1(h,a)=h^{n-1}\sum_{E_j\in[1,1+h]}\Big|\langle\Op_h(a)u_j,u_j\big\rangle-\mu_1(a)\Big|=o(1)\quad\text{as }h\to0,
\end{equation}
where $\Op_h(a)$ is the semiclassical pseudodifferential operator with principal symbol $a\in S(\M)$ (see its definition in Section \ref{sec:SA}).
\end{thm}

One can choose $\lambda_j^2=h^{-2}E_j$ to recover the high-energy version quantum ergodicity in Theorem \ref{thm:QE} and Corollary \ref{cor:QE}. With this connection between the two versions of quantum ergodicity, we see that Question \ref{q:deltaequidinM} actually involves choosing $a=\chi_{B(x_0,r)}$ in \eqref{eq:scQE} with $r=h^\rho$. But the smoothened version of such a function $a$ would belong in nice symbol class only if $\rho\in[0,1/2)$, that is, $S_{h^\rho}(\M)$ (see Section \ref{sec:SA}). So answering Question \ref{q:deltaequidinM} for $\rho\ge1/2$ would be very challenging. (It is in fact already challenging for any polynomial scale $h^\rho$ for $\rho>0$ as our results are all in the logarithmic scales.) 

In negatively curved manifolds, we in fact are able to prove quantum ergodicity for larger classes of symbols than $S(\M)$. To this end, we have the following definition.

\begin{defn}[$\rho$-admissibility]\label{defn:rhoad}
Let $\rho\in[0,1/2)$ and $0<h_0\ll1$. We say that $\delta(h)\ge0$ is $\rho$-admissible if $\delta(h)$ satisfies that $h^\rho\le\delta(h)\le1$ for all $h\in(0,h_0]$.
\end{defn}

\begin{ex}
$\delta(h)=|\log h|^{-\alpha}$ for $\alpha>0$ is $\rho$-admissible for all $\rho\in(0,1/2)$. 
\end{ex}

We then define the symbol class $S_{\delta(h)}(\M)$ as follows.

\begin{defn}[$S_{\delta(h)}(\M)$ symbol classes]
Let $\delta(h)$ be $\rho$-admissible for some $\rho\in[0,1/2)$. We say that $a\in C^\infty(T^*\M)$ belongs to $S_{\delta(h)}(\M)$ if and only if for each multi-indices $\alpha,\beta$, there exists a uniform constant $C_{\alpha,\beta}$ independent of $h$ such that
$$\sup_{x\in\M,\xi\in T_x^*\M}|\partial^\alpha_x\partial^\beta_\xi a|\le C_{\alpha,\beta}\,\delta(h)^{-|\alpha|-|\beta|}\langle\xi\rangle^{-|\beta|},$$
where $\langle\xi\rangle=(1+|\xi|_x^2)^{1/2}$.
\end{defn}

Our main theorem, small scale quantum ergodicity theorem, states that Theorem \ref{thm:scQE} can be extended to symbols in $S_{\delta(h)}(\M)$, given that $\delta(h)\to0$ in a certain logarithmical rate as $h\to0$.
\begin{thm}[Small scale quantum ergodicity]\label{thm:deltaQE}
Let $(\M,g)$ be negatively curved. For
$$\alpha>0\quad\text{and}\quad0\le\beta<1,\quad\text{or}\quad\alpha=0\quad\text{and}\quad\beta=1,$$ 
suppose that $\delta(h)=|\log h|^{-\alpha}$. Then for any ONBE $\{u_j\}_{j=0}^\infty$ and all $a\in S_{\delta(h)}(\M)$,
\begin{equation}\label{eq:deltaQE}
V_2(h,a)=h^{n-1}\sum_{E_j\in[1,1+h]}\Big|\big\langle\Op_h(a)u_j,u_j\big\rangle-\mu_1(a)\Big|^2=O\left(|\log h|^{-\beta}\right)\quad\text{as }h\to0.
\end{equation}
\end{thm}

\begin{rmk}
Notice that when $\alpha=0$, we have $\delta(h)=1$ and $S_{\delta(h)}(\M)=S(\M)$. Then Theorem \ref{thm:deltaQE} recovers \eqref{eq:Zelditch} from \cite{Ze2}.
\end{rmk}

We next consider some special symbols in $S_{\delta(h)}(\M)$ that are used to derive small scale mass equidistribution. 

\begin{defn}[$\delta$-microlocalized symbols]
Let $\delta(h)$ be $\rho$-admissible for some $\rho\in[0,1/2)$. Let $b\in C_0^\infty(\R^n\times\R^{n-1})$. In a local chart that contains $z_0=(x_0,\xi_0)\in S^*\M$, write the local coordinate of $z=(x,\xi)\in T^*\M$ as $(\tilde x,\tilde\xi)\in\R^{2n}$. We say that a compactly supported smooth function $a^b_{z_0}(x,\xi;h)$ is a $\delta$-microlocalized symbol if it is locally defined by
\begin{equation}\label{eq:deltasymbol}
a^b_{z_0}(x,\xi;h)=b\left(\frac{\tilde x-\tilde x_0}{\delta(h)},\frac{\widehat{\tilde\xi}-\tilde\xi_0}{\delta(h)}\right)\phi\left(|\tilde\xi|_x\right),
\end{equation}
where $\hat\eta=\eta/|\eta|_x$, and $\phi\in C^\infty_0\big((-1/2,1/2)\big)$ has small compact support and equals $1$ on $(-1/4,1/4)$.
\end{defn}

\begin{rmk}
A $\delta$-microlocalized symbol is in the symbol class $S_{\delta(h)}(\M)$. Since the base function $b$ is in $C_0^\infty(\R^n\times\R^{n-1})$, the support of $a^b_{z_0}$ restricted on $S^*\M$ shrinks to $z_0\in S^*\M$ at the same rate $\delta(h)$ in every direction; the cutoff function $\phi$ ensures that $a^b_{z_0}$ is extended to a well-defined symbol in a neighborhood of $S^*\M$ in $T^*\M$. We are free to choose the base function $b$ in building $a_{z_0}^b(x,\xi;h)$.
\end{rmk}

Considering $\delta$-microlocalized symbols, small scale quantum ergodicity theorem states
\begin{thm}\label{thm:deltaQEdeltamicrosymbol}
Let $(\M,g)$ be negatively curved. For
$$0<\alpha<\frac{1}{2(2n-1)}\quad\text{and}\quad0\le\beta<1-2\alpha(2n-1),\quad\text{or}\quad\alpha=0\quad\text{and}\quad\beta=1,$$ 
suppose that $\delta(h)=|\log h|^{-\alpha}$. Then for any ONBE $\{u_j\}_{j=0}^\infty$, as $h\to0$,
\begin{equation}\label{eq:deltaQEdeltamicrosymbol}
V_2\left(h,a^b_{z_0}\right)=h^{n-1}\sum_{E_j\in[1,1+h]}\Big|\big\langle\Op_h\left(a^b_{z_0}\right)u_j,u_j\big\rangle-\mu_1\big(a^b_{z_0}\big)\Big|^2=C\,\delta(h)^{2(2n-1)}|\log h|^{-\beta},
\end{equation}
uniformly for $z_0\in S^*\M$, where $a^b_{z_0}$ is defined in \eqref{eq:deltasymbol} and $C$ depends on $b$, $\phi$, and $\M$.
\end{thm}

\begin{rmk}
We have the term $\delta(h)^{2(2n-1)}$ in the right-hand side of \eqref{eq:deltaQEdeltamicrosymbol} because $\mu_1\big(a^b_{z_0}\big)=O\left(\delta(h)^{2n-1}\right)$ in the left-hand side.
\end{rmk}

To consider the mass equidistribution of eigenfunctions in small scales of $\M$ (instead of on $S^*\M$), we need the $\delta$-localized symbols.

\begin{defn}[$\delta$-localized symbols]
Let $\delta(h)$ be $\rho$-admissible for some $\rho\in[0,1/2)$. Let $b\in C_0^\infty(\R^n)$. In a local chart that contains $z_0=(x_0,\xi_0)\in S^*\M$, write the local coordinate of $z=(x,\xi)\in T^*\M$ as $(\tilde x,\tilde\xi)\in\R^{2n}$. We say that a compactly supported smooth function $a^b_{x_0}(x,\xi;h)$ is a $\delta$-localized symbol if it is locally defined by
\begin{equation}\label{eq:deltaMsymbol}
a^b_{x_0}(x,\xi;h)=b\left(\frac{\tilde x-\tilde x_0}{\delta(h)}\right)\phi\left(|\tilde\xi|_x\right).
\end{equation}
\end{defn}

Similar to Theorem \ref{thm:deltaQEdeltamicrosymbol}, we have the following theorem, with the shrinking rate of the symbols slightly better than the one in Theorem \ref{thm:deltaQE}.

\begin{thm}\label{thm:deltaQEinM}
Let $(\M,g)$ be negatively curved. For 
$$0<\alpha<\frac{1}{2n}\quad\text{ and }\quad0\le\beta<1-2\alpha n,\quad\text{or}\quad\alpha=0\quad\text{and}\quad\beta=1,$$ 
suppose that $\delta(h)=|\log h|^{-\alpha}$. Then for any ONBE $\{u_j\}_{j=0}^\infty$, as $h\to0$,
\begin{equation}\label{eq:deltaQEinM}
V_2\big(h,a^b_{x_0}\big)=h^{n-1}\sum_{E_j\in[1,1+h]}\Big|\big\langle\Op_h\left(a^b_{x_0}\right)u_j,u_j\big\rangle-\mu_1\big(a^b_{x_0}\big)\Big|^2=C\,\delta(h)^{2n}|\log h|^{-\beta},
\end{equation}
uniformly for $x_0\in S^*\M$, where $a^b_{x_0}$ is defined in \eqref{eq:deltaMsymbol} and $C$ depends on $b$, $\phi$, and $\M$.
\end{thm}

\begin{rmk}
In a recent preprint \cite{HR1}, Hezari and Rivi\`ere proved Theorem \ref{thm:deltaQEinM} for $V_p\big(h,a^b_{x_0}\big)$ of $p\ge2$ with the same $\delta(h)$ when $b$ is a cutoff function. They then used it to study $L^p$ norm and nodal set estimates of the full density quantum ergodic subsequence of eigenfunctions. We refer to their paper for details; see also Sogge \cite{So2}.
\end{rmk}

Using Theorem \ref{thm:deltaQEinM}, we answer Question \ref{q:deltaequidinM} for any fixed point in the manifold.
\begin{cor}\label{cor:deltaequidatx}
Let $(\M,g)$ be negatively curved and $x_0\in\M$. 
Assume that
$$0\le\alpha<\frac{1}{2n}\quad\text{and}\quad r(\lambda)=(\log\lambda)^{-\alpha}.$$ 
Given any ONBE $\{u_j\}_{j=0}^\infty$ with $\Delta u_j=\lambda_j^2u_j$, there exists a full density subsequence $\{u_{j_k}\}$ such that
\begin{equation}\label{eq:deltaequidatx}
\int_{B(x_0,r_{j_k})}|u_{j_k}|^2\,d\Vol=\frac{\Vol(B(x_0,r_{j_k}))}{\Vol(\M)}+o(r_{j_k}^n)\quad\text{as }k\to\infty
\end{equation}
for $r_{j_k}=r(\lambda_{j_k})=(\log\lambda_{j_k})^{-\alpha}$.
\end{cor} 

Notice that in the above corollary, the subsequence $\{u_{j_k}\}$ depends on the point $x_0$. To get a uniform result, we use a covering argument to get a weaker result as uniform compatibility. In the process, we concede certain loss on the shrinking rate.

\begin{cor}\label{cor:deltaequidinM}
Let $(\M,g)$ be negatively curved. Assume that
$$0\le\alpha<\frac{1}{3n}\quad\text{and}\quad r(\lambda)=(\log\lambda)^{-\alpha}.$$ 
Given any ONBE $\{u_j\}_{j=0}^\infty$ with $\Delta u_j=\lambda_j^2u_j$, there exists a full density subsequence $\{u_{j_k}\}$ such that 
\begin{equation}\label{eq:deltaequidinM}
c\Vol(B\big(x,r_{j_k}))\le\int_{B(x,r_{j_k})}|u_{j_k}|^2\,d\Vol\le C\Vol(B(x,r_{j_k}))\quad\text{as }k\to\infty,
\end{equation}
uniformly for all $x\in\M$, where the positive constants $c$ and $C$ depends only on $\M$.
\end{cor}

\subsection*{Outline and organization}
We prove small scale quantum ergodicity in the same spirit as in Theorems \ref{thm:QE} and \ref{thm:scQE}. The key ingredients are semiclassical trace formula and Egorov's theorem. The former requires to treat symbols in a more general class other than $C^\infty$ ones independent of $h$. The latter, connecting time evolution of classical observables (the symbol $a$ in $T^*\M$) and quantum observables (semiclassical pseudodifferential operator $\Op_h(a)$ on $L^2(\M)$), also concerns symbols depending on $h$. Such correspondence is valid up to the Ehrenfest time $T_E\approx|\log h|$ (Bouzouina and Robert \cite{BR}). In the process, we trace all the dependence on $h$.

Thus, small scale quantum ergodicity is reduced to estimating the time-average of the quantum observable, which is controlled by the time-average of its principal symbol. Liverani \cite{Liv}'s exponential decay of correlations then can be used to give a quantitative estimate on the time-average of the symbol in terms of its H\"older norm; indeed, the time-average has polynomial decay in time. By properly choosing the $\delta(h)$ and thus the symbol class $S_{\delta(h)}(\M)$, we can prove Theorem \ref{thm:deltaQE}. Since Ehrenfest time $T_E$ is of order $|\log h|$, we can only select such symbols with H\"older norm of order $|\log h|^{-\alpha}$ for some $\alpha>0$. This is essentially the reason why we have logarithmically shrinking rates.

We organize this paper as follows. In Sections \ref{sec:SA} and \ref{sec:geo}, we review semiclassical analysis and geodesic flows in negatively curved manifolds, respectively. In Section \ref{sec:deltaQE}, we prove small scale quantum ergodicity theorems, i.e. Theorems \ref{thm:deltaQE}, \ref{thm:deltaQEdeltamicrosymbol}, and \ref{thm:deltaQEinM}. In Section \ref{sec:deltaequidinM}, we prove small scale mass equidistribution in Corollaries \ref{cor:deltaequidatx} and \ref{cor:deltaequidinM}; we also point out some further investigation on small scale quantum ergodicity, particularly on Question \ref{q:deltaequidinM}.

\subsection*{A note on the previous works} 
Besides the above mentioned previous results on quantum ergodicity, there are other exciting development in this area recently. We refer to Zelditch \cite{Ze3} and Sarnak \cite{Sa2} for its current stage.

\section{Semiclassical analysis}\label{sec:SA}
In this section, we review semiclassical analysis that will be used to prove small scale quantum ergodicity. Most of the notations and facts below are fairly standard. We refer to Zworski \cite{Zw} for a complete treatment in this subject. 

\subsection{Phase space}
$\M$ can be a bounded open set in $\R^n$ or a compact Riemannian manifold. An element, called a state, in the cotangent bundle $T^*\M$ is denoted as $z=(x,\xi)$ with $x\in\M$ and $\xi\in T^*_x\M$. 

\subsection{Symbol classes}\label{sec:symbol}
Let $m\in\R$ and $\rho\in[0,1/2)$. The symbol class $S^m_{h^\rho}(\M)$ is defined as follows: $a\in C^\infty(T^*\M)$ belongs to $S^m_{h^\rho}(\M)$ if and only if for each multi-indices $\alpha,\beta$, there exists a uniform constant $C_{\alpha,\beta}$ independent of $h$ such that
$$\sup_{x\in\M,\xi\in T_x^*\M}|\partial^\alpha_x\partial^\beta_\xi a|\le C_{\alpha,\beta}\,h^{-\rho(|\alpha|+|\beta|)}\langle\xi\rangle^{m-|\beta|},$$
where $\langle\xi\rangle=(1+|\xi|_x^2)^{1/2}$.
\begin{enumerate}
\item If $\rho=0$, we denote $S^m_{h^\rho}(\M)$ by $S^m(\M)$.
\item If $m=0$, we denote $S^m_{h^\rho}(\M)$ by $S_{h^\rho}(\M)$.
\item If $a$ has compact support which satisfies the above estimate, we say that $a\in S^\comp_{h^\rho}(\M)$. Clearly, $S^{\comp}_{h^\rho}(\M)\subset S^m_{h^\rho}(\M)$ for all $m\in\R$.
\item We denote $S^{-\infty}_{h^\rho}(\M)=\cap_{m\in\R}S^m_{h^\rho}(\M)$ and $S^\infty_{h^\rho}(\M)=\cup_{m\in\R}S^m_{h^\rho}(\M)$.
\end{enumerate}

These classes are independent of the choice of coordinates in $\M$. Moreover, the seminorms $|\cdot|_{\alpha,\beta}$ in $S_{h^\rho}(\M)$ is defined by the best constant $C_{\alpha,\beta}$ that can be used in the above inequality.

\begin{ex}
Let $\delta(h)$ be $\rho$-admissible. Then $S_{\delta(h)}(\M)\subset S_{h^\rho}(\M)$; moreover, the $\delta$-microlocalized symbols and $\delta$-localized symbols are in $S_{\delta(h)}(\M)\cap S^\comp_{h^\rho}(\M)$.
\end{ex}

\subsection{Quantizations in $\R^n$}\label{sec:QinR}
Every classical observable $a$ in the phase space $T^*\R^n$ corresponds to a quantum observable $\Op_h(a)$ as a semiclassical pseudodifferential operator acting on $L^2$ functions in $\R^n$. In this note, we use Weyl quantization.

\begin{defn}[Weyl quantization]
Given $a\in S^m_{h^\rho}(\R^n)$, $\rho\in[0,1/2)$, we define the Weyl quantization
$$\Op_h(a)u(x)=\frac{1}{(2\pi h)^n}\int_{\R^n}\int_{\R^n}e^{i(x-y)\cdot\eta/h}a\left(\frac{x+y}{2},\eta;h\right)u(y)\,d\eta dy$$
for $u\in\mathcal S(\R^n)$.
\end{defn}

\begin{rmk}
$\Op_h(a)$ is self-adjoint if $a$ is a real-valued symbol.
\end{rmk}

\subsection{Quantization in $\M$}\label{sec:QinM}
We now define $\Psi^m_{h^\rho}(\M)$ of semiclassical pseudodifferential operators with symbols in $S^m_{h^\rho}(\M)$, $\rho\in[0,1/2)$; then we establish the correspondence of $A\in\Psi^m_{h^\rho}(\M)$ and its semiclassical principal symbol $a$. See Zworski \cite[Section 14.2]{Zw} for more details, and also Dyatlov and Guillarmou \cite[Section 3.1]{DyGu} for symbols in $S_{h^\rho}^m(\M)$. The correspondence is one-to-one modulo lower order terms. Denote
$$a=\sigma(A):\Psi^m_{h^\rho}(\M)\to S^m_{h^\rho}(\M)/h^{1-2\rho}S^{m-1}_{h^\rho}(\M),$$
and its right inverse, a non-canonical quantization map for $a\in S^m_{h^\rho}(\M)$:
$$A=\Op_h(a):S^m_{h^\rho}(\M)\to\Psi^m_{h^\rho}(\M).$$
$\sigma(A)$ is called the principal symbol of $A$. It is modulo $h^{1-2\rho}S^{m-1}_{h^\rho}(\M)$ unique under change of quantizations and change of local coordinates. Following the same fashion in Section \ref{sec:symbol},
\begin{enumerate}
\item if $\rho=0$, we denote $\Psi^m_{h^\rho}(\M)$ by $\Psi^m(\M)$;
\item if $m=0$, we denote $\Psi^m_{h^\rho}(\M)$ by $\Psi_{h^\rho}(\M)$;
\item if $a$ has compact support, we say that $\Op(a)\in\Psi^\comp_{h^\rho}(\M)$. $\Psi^{\comp}_{h^\rho}(\M)\subset\Psi^m_{h^\rho}(\M)$ for all $m\in\R$. Moreover, if $A\in\Psi^\comp_{h^\rho}(\M)$, then $A=\Op_h(a)$ for some $a\in S^\comp_{h^\rho}(\M)$;
\item we denote $\Psi^{-\infty}_{h^\rho}(\M)=\cap_{m\in\R}\Psi^m_{h^\rho}(\M)$ and $\Psi^\infty_{h^\rho}(\M)=\cup_{m\in\R}\Psi^m_{h^\rho}(\M)$.
\end{enumerate}
The usual operations involving semiclassical pseudodifferential operators are as follows. Let $A\in\Psi^m_{h^\rho}(\M)$ and $B\in\Psi^{m'}_{h^\rho}(\M)$.
\begin{enumerate}
\item Let $A^\star$ be the adjoint operator of $A$ in $L^2(\M)$. Then
\begin{equation}\label{eq:SDOadjoint}
\sigma(A^\star)=\ol{\sigma(A)}+O_{S^{m-1}_{h^\rho}(\M)}(h^{1-2\rho}).
\end{equation}
\item
\begin{equation}\label{eq:SDOproduct}
\sigma(AB)=\sigma(A)\sigma(B)+O_{S^{m+m'-1}_{h^\rho}(\M)}(h^{1-2\rho}).
\end{equation}
\item
\begin{equation}\label{eq:SDOcommutator}
\sigma([A,B])=-ih\{\sigma(A),\sigma(B)\}+O_{S^{m+m'-2}_{h^\rho}(\M)}(h^{2(1-2\rho)}),
\end{equation}
where $\{\cdot,\cdot\}$ stands for the Poisson bracket defined by
$$\{a,b\}=\frac{\partial a}{\partial x}\frac{\partial b}{\partial\xi}-\frac{\partial a}{\partial\xi}\frac{\partial b}{\partial x}.$$
\end{enumerate}

We also need the following $L^2$ boundedness of semiclassical pseudodifferential operators in $\Psi_{h^\rho}(\M)$, see e.g. \cite[Theorem 4.23]{Zw}.

\begin{thm}[$L^2$ boundedness]\label{thm:SDOL2}
Let $a\in S_{h^\rho}(\M)$ for $\rho\in[0,1/2)$. Then
$$\|\Op_h(a)u\|_{L^2(\M)}\le C\|u\|_{L^2(\M)},$$
where $C$ depends on finite number of seminorms of $a$, and is independent of $h$.
\end{thm}

\subsection{Semiclassical measures}\label{sec:scmeasures}
For each eigenfunction $u_j(h)$ of $h^2\Delta$, we define the distribution associated with $u_j(h)$ as
$$\mathcal W_j(a)=\left\langle\Op_h(a)u_j(h),u_j(h)\right\rangle\quad\text{for }a\in S(\M).$$
We see that $\mathcal W_j$ depends on local coordinates, partition of unity, etc. However, as $h\to0$, the accumulation points of $\mathcal W_j$ are independent of such choices. We define the semiclassical measures as the limit points of $\{\mathcal W_j\}$. Note that $h^2\Delta u_j(h)=E_ju_j(h)$, that is,
$$\left(h^2\Delta-E_j\right)u_j(h)=0.$$
Since the semiclassical symbol of $h^2\Delta-E_j$ is $|\xi|_x^2-E_j$, from Zworski \cite[Theorems 5.3 and 5.4]{Zw}, we have
\begin{thm}\label{thm:Hormander}
Let $0<c_1<c_2<\infty$. Then any semiclassical measure for $E_j\in[c_1,c_2]$ is supported in the energy shell 
$$\left\{(x,\xi)\in T^*\M:H(x,\xi)=|\xi|^2_x\in[c_1,c_2]\right\},$$ 
and is invariant under the geodesic flow $G_t$.
\end{thm}

In particular, the semiclassical measures corresponding to eigenvalues $E_j\in[1,1+h]$ is supported on the energy layer $S^*\M$. From this point of view, quantum ergodicity studies the impact of geodesic flow $G_t$ on the semiclassical measures. Theorem \ref{thm:scQE} asserts that if $G_t$ is ergodic on $S^*\M$, then the difference of $\mathcal W_j$ and the Liouville measure $\mu_1$ converges to $0$ in Ces\`aro summation for $E_j\in[1,1+h]$.

\subsection{Egorov's theorem}\label{sec:Egorov}
The geodesic flow $G_t:(x(0),\xi(0))\to(x(t),\xi(t))$ in $T^*\M$ is generated by the Hamilton equation
$$\frac{dx}{dt}=\frac{\partial H}{\partial\xi}(x,\xi),\quad\frac{d\xi}{dt}=-\frac{\partial H}{\partial x}(x,\xi),$$
in which $H(x,\xi)=|\xi|_x^2$. Then the time evolution of a classical symbol $a$ satisfies
$$\frac{d}{dt}a(G_t(x,\xi))=\{H,a\}(G_t(x,\xi)),$$
where $\{H,a\}$ is the Poisson bracket. The quantum time evolution of $\Op_h(a)$ is associated with the unitary Fourier integral operator $U(t)=e^{-ith\Delta}$, called the Schr\"odinger propagator of $h^2\Delta$. Egorov's theorem states that the quantum time evolution of $\Op_h(a)$ can be approximated by the quantization of classical time evolution of $a$, within finite time:
$$U(-t)\circ\Op_h(a)\circ U(t)\approx\Op_h(a\circ G_t).$$
Such correspondence connects the classical observable $a$ and the quantum observable $\Op_h(a)$ under time evolution. Precisely, if $a\in S(\M)$ and $|t|\le T<\infty$, then
$$\Big\|U(-t)\circ\Op_h(a)\circ U(t)-\Op_h(a\circ G_t)\Big\|_{L^2(\M)\to L^2(\M)}=O(h)\quad\text{as }h\to0.$$

\subsection{Semiclassical trace formula}
Here we use Schubert \cite[Proposition 1]{Sc} in the present context.
\begin{prop}\label{prop:STF}
Suppose that $\psi$ is a smooth function on $\R$ such that the Fourier transform $\hat\psi$ has compact support in a small neighbourhood of $0$ which contains no period of a periodic orbit of the geodesic flow $G_t$ on $S^*\M$. Then for every $a\in S(\M)$, we have
$$\left|h^{n-1}\sum_j\psi\left(\frac{1-E_j}{h}\right)\left\langle\Op_h(a)u_j,u_j\right\rangle-\frac{\hat\psi(0)}{(2\pi)^n}\mu_1(a)\right|\le Ch\|a\|_{C^{2n+8}},$$
where the constant $C$ depends on $\psi$ and $\M$.
\end{prop}

\section{Geodesic flows in negatively curved manifolds}\label{sec:geo}
In this section, we gather some facts on the geodesic flow $G_t$ in a negatively curved manifold $(\M,g)$. Recall that $H(x,\xi)=|\xi|^2_x$. $G_t$ is Anosov on $S^*\M$, that is, the tangent bundle $TS^*\M$ splits into $G_t$-invariant sub-bundles 
$$E^u(v)\oplus E^s(v)\oplus\R X_H(v)\quad\text{for }v\in TS^*\M.$$ 
Here, $E^u$ and $E^s$ are the unstable and stable subspaces, respectively. They satisfy
\begin{equation}\label{eq:Anosov}
\begin{cases}
\|dG_tv\|\le Ce^{-kt}\|v\| & \forall v\in E^s,t\ge0;\\
\|dG_tv\|\le Ce^{kt}\|v\| & \forall v\in E^u,t\le0,
\end{cases}\quad\text{for some }k>0,
\end{equation}
where $G_t:z\to G_t(z)$, $dG_t:T_zS^*\M\to T_{G_t(z)}S^*\M$ is the differential, and $\|\cdot\|$ is the norm defined in $T^*\M$ (e.g. by the Sasaki metric, c.f. Ballmann \cite{B}). The sub-bundles are integrable and induce stable and unstable foliations. We refer to Katok and Hasselblatt \cite{KH} for background. 

\subsection{Egorov's theorem until the Ehrenfest time}\label{sec:Ehrenfest}
In a negatively curved manifold $(\M,g)$, we use the above information on its geodesic flow to describe the long-time evolution $a\circ G_t$ of the classical observable $a$ and its quantum long-time evolution in the Egorov's theorem. See Bouzouina and Robert \cite{BR} in a more general setting.  

Let $\mathcal E=\{(x,\xi)\in T^*\M:1/2\le|\xi|_x\le3/2\}$ be an energy shell in $T^*\M$. So $\mathcal E\supset S^*\M$. Also, $G_t$ preserves $\mathcal E$. From Anantharaman and Nonnenmacher \cite[Section 5.2]{AN} that
$$\sup_{z\in\mathcal E}|\partial_z^\alpha G_t(z)|\le C_\alpha e^{l|\alpha||t|}\quad\text{for }t\in\R,$$
where $l$ can be chosen as a number that is greater than the maximal expansion rate of $G_t$ in $\mathcal E$. Now consider a symbol $a\in S^\comp_{h^\rho}(\M)$, $\rho\in[0,1/2)$, compactly supported in $\mathcal E$. Clearly $a\circ G_t$ is still compactly supported in $\mathcal E$. Moreover, we have
$$\sup_{z\in\mathcal E}|\partial_z^\alpha(a\circ G_t)(z)|\le C_{a,\alpha}e^{l|\alpha||t|}h^{-\rho|\alpha|}\quad\text{for }t\in\R.$$
If $\epsilon\in[0,1/2-\rho)$, setting
$$T_E=\frac{|\log h|}{l},$$
then 
$$\sup_{z\in\mathcal E}|\partial_z^\alpha(a\circ G_t)(z)|\le C_{a,\alpha}h^{-(\rho+\epsilon)|\alpha|}\quad\text{for }|t|\le\epsilon T_E.$$
This means that $a\circ G_t\in S^\comp_{\rho+\epsilon}(\M)$ if $|t|\le\epsilon T_E$. Since $\rho+\epsilon\in[0,1/2)$, $a\circ G_t$ is still in a nice symbol class. We call $T_E$ the Ehrenfest time.
Then Egorov's theorem in Section \ref{sec:Egorov}, connecting the time evolution of a classical observable $a$ and its quantum counterpart $\Op_h(a)$, can be extended to more general symbols until Ehrenfest time. See Anantharaman \cite[Theorem 4.2.4]{Ana}. 
\begin{thm}[Egorov's theorem until the Ehrenfest time]\label{thm:Egorov}
Let $a\in S^\comp_{h^\rho}(\M)$, $\rho\in[0,1/2)$, compactly supported in $\mathcal E$. Set $\epsilon\in[0,1/2-\rho)$. Then
$$\sup_{|t|\le\epsilon T_E}\Big\|U(-t)\circ\Op_h(a)\circ U(t)-\Op_h(a\circ G_t)\Big\|_{L^2(\M)\to L^2(\M)}=O(h^{1-2\rho-2\epsilon})$$
as $h\to0$.
\end{thm}

\subsection{Rate of ergodicity}
In a negatively curved manifold $(\M,g)$, the geodesic flow $G_t$ on $S^*\M$ is Anosov, hence is ergodic with respect to the Liouville measure $\mu_1$. See Anosov \cite{Ano}. Let $f\in L^2(S^*\M)$. Define the time-average of $f$ as
$$\Av_T(f)=\frac1T\int_0^Tf\circ G_t\,dt.$$
Then the von Neumann mean ergodic theorem states that 
$$\|\Av_T(f)-\mu_1(f)\|_{L^2(S^*\M)}=o_{f,T}(1)\quad\text{as }T\to\infty.$$
To get a more quantitative version of the above convergence, we need the exponential decay of correlations from Liverani \cite[Corollary 2.5]{Liv}:
\begin{thm}[Exponential decay of correlations]\label{thm:rom}
For each $\gamma\in(0,1)$, there exist $c,C>0$ depending on $\gamma$ such that for each $f,g\in C^\gamma(S^*\M)$,
\begin{equation}\label{eq:rom}
\left|\int_{S^*\M}f\,g\circ G_t\,d\mu_1-\mu_1(f)\mu_1(g)\right|\le C e^{-ct}\|f\|_\gamma\|g\|_\gamma. 
\end{equation}
\end{thm}

Here, $\|f\|_\gamma$ is the H\"older norm of a function $f\in C^\gamma(S^*\M)$, the space of H\"older continuous functions on $S^*\M$. From this result we derive the following theorem.

\begin{thm}[Rate of ergodicity]\label{thm:roe}
Let $\gamma\in(0,1)$ and $f\in C^\gamma(S^*\M)$. There exist $C>0$ depending on $\gamma$ and $\M$ such that,
\begin{equation}\label{eq:roe}
\|\Av_T(f)-\mu_1(f)\|_{L^2(S^*\M)}\le\frac{C\|f\|_\gamma}{\sqrt T}.
\end{equation}
\end{thm}

\begin{rmk}
The left-hand side of \eqref{eq:roe} is the variance in the central limiting theorem (see Zelditch \cite{Ze2}). So the above theorem asserts that the variance tends to zero in a polynomial rate of time $T^{-1/2}$ for H\"older continuous functions.
\end{rmk}

\begin{proof}[Proof of Theorem \ref{thm:roe}]
Without loss of generality, we may assume that $\mu_1(f)=0$. Compute that
\begin{eqnarray*}
&&\|\Av_T(f)\|_{L^2(S^*\M)}^2\\
&=&\int_{S^*\M}\left(\frac1T\int_0^Tf\circ G_t\,dt\right)\left(\frac1T\int_0^T\ol f\circ G_s\,ds\right)\,d\mu_1\\
&=&\frac{1}{T^2}\int_0^T\int_0^T\int_{S^*\M}f\circ G_t\,\ol f\circ G_s\,d\mu_1\,dt\,ds\\
&=&\frac{1}{T^2}\int_0^T\int_0^T\int_{S^*\M}\left(f\circ G_{t}\,\ol f\circ G_{s}\right)\circ G_{-s}\,d\mu_1\,dt\,ds\quad\text{since $G_{-s}$ is symplectic}\\
&=&\frac{1}{T^2}\int_0^T\int_0^T\int_{S^*\M}f\circ G_{t-s}\,\ol f\,d\mu_1\,dt\,ds\\
&\le&\frac{C\|f\|^2_\gamma}{T^2}\int_0^T\int_0^Te^{-c|t-s|}\,dt\,ds\quad\quad\text{by Theorem \ref{thm:rom}}\\
&\le&\frac{C\|f\|^2_\gamma}{T},
\end{eqnarray*}
and the theorem follows.
\end{proof}

\section{small scale quantum ergodicity: Proofs of Theorem \ref{thm:deltaQE}, \ref{thm:deltaQEdeltamicrosymbol}, and \ref{thm:deltaQEinM}}\label{sec:deltaQE}
In this section, we prove small scale quantum ergodicity. 

\subsection{Proof of Theorem \ref{thm:deltaQE}} 
Let $\delta(h)$ be $\rho$-admissible for some $\rho\in[0,1/2)$.

\noindent\textbf{Step 1}. Given $a\in S_{\delta(h)}(\M)$, we first assume that 
$$\text{$a$ is compactly supported in $\mathcal E$ and $\mu_1(a)=0$}.$$
Recall that $\mathcal E=\{1/2\le|\xi|_x\le3/2\}$. Then $a\in S^\comp_{h^\rho}(\M)$ so the results in Section \ref{sec:Ehrenfest} apply. That is, $a\circ G_t\in S^\comp_{\rho+\epsilon}(\M)$ if $\epsilon\in[0,1/2-\rho)$ and $|t|\le\epsilon T_E$. Therefore, the time average $\Av_T(a)\in S^\comp_{\rho+\epsilon}(\M)$ for $|T|\le\epsilon T_E$.

Next, we see that for any $\gamma\in(0,1)$,
$$\|a\|_{\gamma}\le C\delta(h)^{-\gamma},$$
where $C$ depends on finite number of seminorms of $a$. Then the rate of ergodicity in Theorem \ref{thm:roe} implies that
$$\|\Av_T(a)\|_{L^2(S^*\M)}\le\frac{C\|a\|_{\gamma}}{\sqrt T}\le\frac{C\delta(h)^{-\gamma}}{\sqrt T}.$$
We set $T=\epsilon T_E$. The above inequality continues as
\begin{equation}\label{eq:roedeltasymbol}
\|\Av_{\epsilon T_E}(a)\|_{L^2(S^*\M)}\le C\frac{\delta(h)^{-\gamma}}{\sqrt{\epsilon|\log h|/l}}\le C\sqrt{\frac{l}{\epsilon}}\frac{\delta(h)^{-\gamma}}{\sqrt{|\log h|}}.
\end{equation}
For the semiclassical pseudodifferential operator $\Op_h(a)$, we define
$$\Av^U_T(\Op_h(a))=\frac1T\int_0^TU(-t)\Op_h(a)U(t)\,dt.$$
To evaluate $V_2(h,a)$ in Theorem \ref{thm:deltaQE}, note that 
$$U(t)u_j=e^{-ith\Delta}u_j=e^{-itE_j/h}u_j,$$
then
$$\big\langle\Av^U_T(\Op_h(a))u_j,u_j\big\rangle=\big\langle\Op_h(a)u_j,u_j\big\rangle\quad\text{for all }t\in\R.$$
Compute that
\begin{eqnarray}
&&V_2(h,a)\nonumber\\
&=&h^{n-1}\sum_{E_j\in[1,1+h]}\Big|\big\langle\Op_h(a)u_j,u_j\big\rangle\Big|^2\nonumber\\
&=&h^{n-1}\sum_{E_j\in[1,1+h]}\Big|\big\langle\Av^U_T(\Op_h(a))u_j,u_j\big\rangle\Big|^2\nonumber\\
&\le&h^{n-1}\sum_{E_j\in[1,1+h]}\big\langle\Av^U_T(\Op_h(a))^\star\,\Av^U_T(\Op_h(a))u_j,u_j\big\rangle\label{eq:V2ha}
\end{eqnarray}
by Cauchy-Schwarz inequality, where we drop the absolute value sign since each term in the last summation is non-negative. 

Since $\Av_T(a)\in S^\comp_{\rho+\epsilon}(\M)$ for $\epsilon\in[0,1/2-\rho)$ and $|t|\le\epsilon T_E$, using Egorov's theorem until Ehrenfest time in Theorem \ref{thm:Egorov}, 
\begin{eqnarray*}
&&\Big\|\Av^U_{\epsilon T_E}(\Op_h(a))-\Op_h(\Av_{\epsilon T_E}(a))\Big\|_{L^2(\M)\to L^2(\M)}\\
&\le&\sup_{|t|\le\epsilon T_E}\Big\|U(-t)\circ\Op_h(a)\circ U(t)-\Op_h(a\circ G_t)\Big\|_{L^2(\M)\to L^2(\M)}\\
&=&O(h^{1-2\rho-2\epsilon}).
\end{eqnarray*}
Together with the similar estimate on the adjoint, $\Av_T(\Op_h(a))^\star$, we have
\begin{eqnarray*}
&&\Big\|\Av^U_{\epsilon T_E}(\Op_h(a))^\star\,\Av^U_{\epsilon T_E}(\Op_h(a))-\Op_h(\ol{\Av_{\epsilon T_E}(a}))\Op_h(\Av_{\epsilon T_E}(a))\Big\|_{L^2(\M)\to L^2(\M)}\\
&=&O(h^{1-2\rho-2\epsilon}),
\end{eqnarray*}
using the $L^2$ boundedness in Theorem \ref{thm:SDOL2} of operators in $\Psi_{h^\rho}(\M)$. We see that using \eqref{eq:SDOproduct}
$$\Op_h(\ol{\Av_{\epsilon T_E}(a)})\Op_h(\Av_{\epsilon T_E}(a))\in\Psi_{\rho+\epsilon}^\comp(\M)$$ 
with principal symbol 
$$|\Av_{\epsilon T_E}(a)|^2\in S^\comp_{\rho+\epsilon}(\M).$$
Continue to estimate $V_2(h,a)$. By Weyl's law (c.f. Duistermaat and Guillemin \cite{DuGu}) that $\#\{E_j\in[1,1+h]\}=O(h^{1-n})$, we have that
\begin{eqnarray*}
\eqref{eq:V2ha}&\le&h^{n-1}\sum_{E_j\in[1,1+h]}\big\langle\Av^U_T(\Op_h(a))^\star\,\Av^U_T(\Op_h(a))u_j,u_j\big\rangle\\
&=&h^{n-1}\sum_{E_j\in[1,1+h]}\big\langle\Op_h(\ol{\Av_{\epsilon T_E}(a}))\Op_h(\Av_{\epsilon T_E}(a))u_j,u_j\big\rangle+O(h^{1-2\rho-2\epsilon})\\
&\le&h^{n-1}\sum_j\psi\left(\frac{1-E_j}{h}\right)\big\langle\Op_h(\ol{\Av_{\epsilon T_E}(a}))\Op_h(\Av_{\epsilon T_E}(a))u_j,u_j\big\rangle+O(h^{1-2\rho-2\epsilon}).
\end{eqnarray*}
Here, we choose the positive function $\psi$ such that $\hat\psi\in C^\infty_0(\R)$ and $\psi\ge1$ in $[-1,0]$ so
$$\psi\left(\frac{1-E_j}{h}\right)=1\quad\text{when }E_j\in[1,1+h];$$
moreover, we can always assume that $\supp\hat\psi$ contains no periods of $G_t$ (since $\M$ has injectivity radius much greater than $1$). 

Then applying semiclassical trace formula Proposition \ref{prop:STF}, we have
\begin{eqnarray*}
\eqref{eq:V2ha}&=&h^{n-1}\sum_j\psi\left(\frac{1-E_j}{h}\right)\big\langle\Op_h(\ol{\Av_{\epsilon T_E}(a}))\Op_h(\Av_{\epsilon T_E}(a))u_j,u_j\big\rangle+O(h^{1-2\rho-2\epsilon})\\
&=&\mu_1\left(|\Av_{\epsilon T_E}(a)|^2\right)+Ch\left\||\Av_{\epsilon T_E}(a)|^2\right\|_{C^{2n+8}}+O(h^{1-2\rho-2\epsilon})\\
&=&\mu_1\left(|\Av_{\epsilon T_E}(a)|^2\right)+O\left(h^{1-2(2n+8)(\rho+\epsilon)}\right).
\end{eqnarray*}
From \eqref{eq:roedeltasymbol} we have
\begin{equation}\label{eq:spaceaverage}
\mu_1\big(|\Av_{\epsilon T_E}(a)|^2\big)=\|\Av_{\epsilon T_E}(a)\|_{L^2(S^*\M)}^2\le C\frac{l}{\epsilon}\frac{\delta(h)^{-2\gamma}}{|\log h|}.
\end{equation}
The theorem evidently follows when $\alpha=0$ and $\beta=1$. If
$$\alpha>0\quad\text{and}\quad\delta(h)=|\log h|^{-\alpha},$$
then $a\in S^\comp_{\delta(h)}(\M)\subset S^\comp_{h^\rho}(\M)$ for all $\rho\in(0,1/2)$. We therefore derive that
$$V_2(h,a)=C\mu_1\big(|\Av_{\epsilon T_E}(a)|^2\big)+O\left(h^{1-2(2n+8)(\rho+\epsilon)}\right)=C|\log h|^{2\alpha\gamma-1}\quad\text{as }h\to0$$
by picking $\rho$ and $\epsilon$ small enough. Notice that the rate of ergodicity in Theorem \ref{thm:roe} is valid for all $\gamma\in(0,1)$. Hence, for
$$0\le\beta<1,$$
one can find $\gamma\in(0,1)$ such that
$$\beta\le1-2\alpha\gamma<1.$$
Therefore,
$$|\log h|^{2\alpha\gamma-1}=O(|\log h|^{-\beta}).$$

\noindent\textbf{Step 2}. Let $a\in S_{\delta(h)}(\M)$, where $\delta(h)=|\log h|^{-\alpha}$ for $\alpha\ge0$. We have that $a\in S_{h^\rho}(\M)$ for all $\rho\in(0,1/2)$. Consider the symbol 
$$\tilde a=(a-\mu_1(a))\phi(|\xi|_x).$$ 
Then $\tilde a$ is compactly supported in the energy shell $\mathcal E$ since $\supp\phi\subset(-1/2,1/2)$; moreover
$$\mu_1(\tilde a)=\int_{S^*\M}(a-\mu_1(a))\phi(|\xi|_x)\,d\mu_1=\int_{S^*\M}a\,d\mu_1-\mu_1(a)=0,$$
because $\phi(|\xi|_x)=1$ near $S^*\M$. So Theorem \ref{thm:deltaQE} applies to the symbol $\tilde a$ by the previous step:
$$V_2(h,\tilde a)=h^{n-1}\sum_{E_j\in[1,1+h]}\Big|\big\langle\Op_h(\tilde a)u_j,u_j\big\rangle\Big|^2=O\left(|\log h|^{-\beta}\right)\quad\text{as }h\to0.$$
To prove Theorem \ref{thm:deltaQE} for the symbol $a$, note the difference
\begin{eqnarray*}
&&\left|\big\langle\Op_h(\tilde a)u_j,u_j\big\rangle-\left(\big\langle\Op_h(a)u_j,u_j\big\rangle-\mu_1(a)\right)\right|\\
&=&\left|\Big\langle\Op_h\big((a-\mu_1(a))\phi(|\xi|_x)\big)u_j,u_j\Big\rangle\right|\\
&\le&\left\|\Op_h\big((a-\mu_1(a))(\phi(|\xi|_x)-1)\big)u_j\right\|_{L^2(\M)}:=\|Qu_j\|_{L^2(\M)},
\end{eqnarray*}
in which the semiclassical pseudodifferential operator $Q$ has symbol
$$q=(a-\mu_1(a))(\phi(|\xi|_x)-1)\in S_{h^\rho}(\M).$$
We see that $q$ is supported in $T^*\M\setminus\mathcal E$. We also know that the symbol of $h^2\Delta-E_j$, $p=|\xi|^2_x-E_j$, when $E_j\in[1,1+h]$ is elliptic in $T^*\M\setminus\mathcal E$, that is, $|\xi|^2_x-E_j\ge c>0$. Let $r=q/p$ and $R=\Op_h(r)$. Then $r\in S_{h^\rho}(\M)$ and by \eqref{eq:SDOproduct}
$$Q=RP+E,\quad\text{where }E\in h^{1-2\rho}S_{h^\rho}(\M).$$
Note that the operators in the above equation depend on the energy $E_j$. Since $Pu_j=0$, we have from Theorem \ref{thm:SDOL2}
$$\|Qu_j\|_{L^2(\M)}\le\|RPu_j\|_{L^2(\M)}+\|Eu_j\|_{L^2(\M)}=O(h^{1-2\rho}).$$
Now we compute that
\begin{eqnarray*}
&&V_2(h,a)\\
&=&h^{n-1}\sum_{E_j\in[1,1+h]}\Big|\big\langle\Op_h(a)u_j,u_j\big\rangle-\mu_1(a)\Big|^2\\
&\le&h^{n-1}\sum_{E_j\in[1,1+h]}\left(\left|\big\langle\Op_h(a)u_j,u_j\big\rangle-\mu_1(a)-\big\langle\Op_h(\tilde a)u_j,u_j\big\rangle\right|+\left|\big\langle\Op_h(\tilde a)u_j,u_j\big\rangle\right|\right)^2\\
&\le&h^{n-1}\sum_{E_j\in[1,1+h]}\left(\left|\big\langle\Op_h(\tilde a)u_j,u_j\big\rangle\right|+O(h^{1-2\rho})\right)^2\\
&=&O\left(|\log h|^{-\beta}\right)\quad\text{as }h\to0.
\end{eqnarray*}

\subsection{Proof of Theorem \ref{thm:deltaQEdeltamicrosymbol}}
Notice that for the symbol $a^b_{z_0}$ defined in \eqref{eq:deltasymbol},
\begin{equation}\label{eq:abz0Holder}
\|a^b_{z_0}\|_{\gamma}\le C\delta(h)^{-\gamma},
\end{equation}
where $C$ depends on $b$, $\phi$, and $\M$, and is uniform in $z_0\in S^*\M$.

Following the argument in the previous subsection, 
$$V_2\big(h,a^b_{z_0}\big)\le C\frac{\delta(h)^{-2\gamma}}{|\log h|}.$$
The theorem evidently follows when $\alpha=0$ and $\beta=1$. If
$$0<\alpha<\frac{1}{2n-1}\quad\text{and}\quad\delta(h)=|\log h|^{-\alpha},$$
for
$$0\le\beta<1-2\alpha(2n-1),$$
one can find $\gamma\in(0,1)$ such that
$$\beta\le1-2\alpha(2n+\gamma-1),$$
so
$$\delta(h)^{2(2n-1)}\,|\log h|^{-\beta}\ge|\log h|^{2\alpha\gamma-1}=\frac{\delta(h)^{-2\gamma}}{|\log h|}.$$
Therefore,
$$V_2\big(h,a^b_{z_0}\big)\le C\frac{\delta(h)^{-2\gamma}}{|\log h|}=C\,\delta(h)^{2(2n-1)}|\log h|^{-\beta}.$$ 
Notice that \eqref{eq:abz0Holder} is uniform for $z_0\in S^*\M$. The theorem thus follows since $C$ depends on $\M$ and finite number of derivatives of $b$ and $\phi$. 

\subsection{Proof of Theorem \ref{thm:deltaQEinM}}
The theorem follows by noticing that $a^b_{x_0}$ defined in \eqref{eq:deltasymbol},
\begin{equation}\label{eq:abx0Holder}
\|a^b_{x_0}\|_{\gamma}\le C\delta(h)^{-\gamma},
\end{equation}
where $C$ depends on $b$, $\phi$, and $\M$, and is uniform in $x_0\in\M$. Therefore, when $\alpha$ and $\beta$ in the required ranges, we have that
$$V_2\big(h,a^b_{x_0}\big)\le C\frac{\delta(h)^{-2\gamma}}{|\log h|}=C\,\delta(h)^{2n}|\log h|^{-\beta}.$$ 
Notice that \eqref{eq:abx0Holder} is uniform for $x_0\in\M$. The theorem thus follows since $C$ depends on $\M$ and finite number of derivatives of $b$ and $\phi$.

\section{Small scale mass equidistribution: Proofs of Corollaries \ref{cor:deltaequidatx} and \ref{cor:deltaequidinM}}\label{sec:deltaequidinM}
In this section, we prove Corollaries \ref{cor:deltaequidatx} and \ref{cor:deltaequidinM}. Then we discuss some potential improvements on small scale quantum ergodicity.

\subsection{Proof of Corollary \ref{cor:deltaequidatx}}
We keep the notations that
$$h=\lambda^{-1},\quad h^{-2}E_j=\lambda^2_j,\quad\text{and}\quad\delta(h)=r(\lambda)=(\log\lambda)^{-\beta}$$
such that \eqref{eq:deltaQEinM} holds. Then
$$\Delta u_j=h^{-2}E_ju_j=\lambda^2_ju_j.$$
We use the local coordinates induced by exponential map $\Exp_{x_0}:T_{x_0}\M\to\M$. From the assumption that the injectivity radius of $\M$ is much greater than $1$, $f=\Exp_{x_0}|_{D(0,2)}$ defines a local coordinate system around $x_0$, where $D(0,2)$ is the ball centered at $0$ with radius $2$ in $\R^n$. We still use $\tilde x$ as the the local coordinate of $x$ when $x\in B(x_0,2)$. So $\tilde x_0=0$.

Given a function $b\in C^\infty_0(D(0,2))$, consider the symbol
$$a^b_{x_0}(x,\xi;h)=b\left(\frac{\tilde x}{\delta(h)}\right)\phi\left(|\tilde\xi|_x\right).$$
We have
$$\left\langle\Op_h\left(a^b_{x_0}\right)u_j,u_j\right\rangle=\int_{D(0,2)}b\left(\frac{\tilde x}{\delta(h)}\right)\left|u_j\circ f(\tilde x)\right|^2|Df(\tilde x)|\,d\tilde x+O(h),$$
in which the $O(h)$ term comes from removing the frequency cutoff $\phi$; and
$$\mu_1\left(a^b_{x_0}\right)=\int_{S^*B(0,2)}a^b_{x_0}(x,\xi;h)\,d\mu_1=\frac{1}{\Vol(\M)}\int_{D(0,2)}b\left(\frac{\tilde x}{r(\lambda)}\right)|Df(\tilde x)|\,d\tilde x.$$
Putting them together with \eqref{eq:deltaQEinM}, we have
\begin{eqnarray*}
&&\frac{1}{\#\{\lambda_j\in[\lambda,\sqrt{\lambda^2+\lambda}]\}}\times\\
&&\sum_{\lambda_j\in[\lambda,\sqrt{\lambda^2+\lambda}]}\left|\int_{D(0,2)}b\left(\frac{\tilde x}{r(\lambda)}\right)\left|u_j\circ f(\tilde x)\right|^2|Df(\tilde x)|\,d\tilde x-\frac{1}{\Vol(\M)}\int_{D(0,2)}b\left(\frac{\tilde x}{r(\lambda_j)}\right)|Df(\tilde x)|\,d\tilde x\right|^2\\
&\le&C\,r(\lambda)^{2n}|\log\lambda|^{-\beta},
\end{eqnarray*}
where $C$ depends on $\M$ and finite number of derivatives of $b$ and $\phi$ so the above quantity is bounded by
$$C_{\M,\phi}\|b\|_{C^m}\,r(\lambda)^{2n}|\log\lambda|^{-\beta},$$
where $m>0$ is a sufficiently large integer and $C_{\M,\phi}$ now only depends on $\M$ and $\phi$.

Thus a standard argument (c.f. \cite[Lemma 3]{Ze3}) gives a full density subsequence $\{u_{j_k}\}\subset\{u_j\}$ such that
\begin{eqnarray}
&&\left|\int_{D(0,2)}b\left(\frac{\tilde x}{r(\lambda_{j_k})}\right)|u_{j_k}\circ f(\tilde x)|^2|Df(\tilde x)|\,d\tilde x-\frac{1}{\Vol(\M)}\int_{D(0,2)}b\left(\frac{\tilde x}{r(\lambda_{j_k})}\right)|Df(\tilde x)|\,d\tilde x\right|\nonumber\\
&\le&C_{\M,\phi}\|b\|_{C^m}\,r(\lambda_{j_k})^n|\log\lambda_{j_k}|^{-\tilde\beta}\quad\text{as }k\to\infty\label{eq:1densityb}
\end{eqnarray}
for some $0<\tilde\beta<\beta$. 

For $i\in\mathbb Z$ and $|i|\ge10$, define $b_i\subset C^\infty_0(D(0,2))$ such that $0\le b_i\le1$ and
$$b_i(\tilde x)=\begin{cases}
1 & \text{for }\tilde x\in D(0,1+1/i);\\
0 & \text{for }\tilde x\in D(0,2)\setminus D(0,1+2/i).
\end{cases}$$
For each $b_i$, there is a full density subsequence $\{u^{(i)}_{j_k}\}$ depending on $b_i$ such that \eqref{eq:1densityb} holds. Using a diagonal argument (c.f. \cite[Theorem 15.5]{Zw}), we can find a full density subsequence $\{u_{j_k}\}\subset\{u_j\}$ such that
\begin{eqnarray*}
&&\left|\int_{D(0,2)}b_i\left(\frac{\tilde x}{r(\lambda_{j_k})}\right)|u_{j_k}\circ f(\tilde x)|^2|Df(\tilde x)|\,d\tilde x-\frac{1}{\Vol(\M)}\int_{D(0,2)}b_i\left(\frac{\tilde x}{r(\lambda_{j_k})}\right)|Df(\tilde x)|\,d\tilde x\right|\\
&\le&C_{\M,\phi}\|b_i\|_{C^m}\,r(\lambda_{j_k})^n|\log\lambda_{j_k}|^{-\tilde\beta}\quad\text{as }k\to\infty
\end{eqnarray*}
for all $b_i$. Denote $r_{j_k}=r(\lambda_{j_k})$.  Using the above inequality, we have that for $i\ge10$
\begin{eqnarray*}
&&\int_{B(x_0,r_{j_k})}|u_{j_k}|^2\,d\Vol\\
&=&\int_{D(0,r_{j_k})}|u_{j_k}\circ f(\tilde x)|^2|Df(\tilde x)|\,d\tilde x\\
&\le&\int_{D(0,2)}b_i\left(\frac{\tilde x}{r_{j_k}}\right)|u_{j_k}\circ f(\tilde x)|^2|Df(\tilde x)|\,d\tilde x\\
&\le&\frac{1}{\Vol(\M)}\int_{D(0,2)}b_i\left(\frac{\tilde x}{r_{j_k}}\right)|Df(\tilde x)|\,d\tilde x+C_{\M,\phi}\|b_i\|_{C^m}\,r_{j_k}^n|\log\lambda_{j_k}|^{-\tilde\beta}\\
&\le&\frac{\Vol(B(x_0,r_{j_k}))}{\Vol(\M)}+C_\M r_{j_k}^n/i+C\|b_i\|_{C^m}\,r_{j_k}^n|\log\lambda_{j_k}|^{-\tilde\beta}.
\end{eqnarray*}
The last inequality follows from the fact that
$$\int_{D(0,2)\setminus D(0,1)}b_i\left(\frac{\tilde x}{r_{j_k}}\right)|Df(\tilde x)|\,d\tilde x\le\int_{D(0,1+2/i)\setminus D(0,1)}|Df(\tilde x)|\,d\tilde x\le C_\M r_{j_k}^n/i,$$
where $C_\M$ depends only on $\M$. We know that
$$\|b_i\|_{C^m}\le c\,i^m,\quad\text{where $c$ is an absolute constant.}$$
So
\begin{eqnarray*}
&&\int_{B(x_0,r_{j_k})}|u_{j_k}|^2\,d\Vol\\
&\le&\frac{\Vol(B(x_0,r_{j_k}))}{\Vol(\M)}+C_\M r_{j_k}^n/i+C_{\M,\phi}\|b_i\|_{C^m}\,r_{j_k}^n|\log\lambda_{j_k}|^{-\tilde\beta}\\
&\le&\frac{\Vol(B(x_0,r_{j_k}))}{\Vol(\M)}+C_{\M,\phi}r_{j_k}^n\left(i^{-1}+i^m|\log\lambda_{j_k}|^{-\tilde\beta}\right).\\
\end{eqnarray*}
Letting 
$$i=\left[\left|\log\lambda_{j_k}\right|^{\frac{\tilde\beta}{m+1}}\right]$$
gives that
$$\int_{B(x_0,r_{j_k})}|u_{j_k}|^2\,d\Vol\le\frac{\Vol(B(x_0,r_{j_k}))}{\Vol(\M)}+o(r_{j_k}^n).$$

Similarly, for $i\le-10$ choose appropriate $b_i$. We can show that
$$\int_{B(x_0,r_{j_k})}|u_{j_k}|^2\,d\Vol\ge\frac{\Vol(B(x_0,r_{j_k}))}{\Vol(\M)}+o(r_{j_k}^n).$$
Therefore, we prove \eqref{eq:deltaequidatx} for this fixed point $x_0\in\M$.

\subsection{Proof of Corollary \ref{cor:deltaequidinM}}\hfill\\
\noindent\textbf{Step 1}. First we need a covering lemma that is similar to Colding and Minicozzi \cite[Lemma 2]{CM}. For readers' convenience, we provide a proof here.
\begin{lemma}\label{lemma:covering}
Let $r>0$ be sufficiently small. Then there exists a family of geodesic balls that covers $\M$:
$$\{B(x_i,r)\}_{i=1}^{N}\supset\M\quad\text{with }N\le c_1r^{-n};$$ 
furthermore, the following statements are valid.
\begin{enumerate}[(1).]
\item Each ball $B(x,r)\subset\M$ is covered by at most $c_2$ balls among $\{B(x_i,r)\}_{i=1}^{N}$:
$$B(x,r)\subset\bigcup_{m=1}^{c_2}B(x_{i_m},r)\quad\text{ for some }\{i_1,...,i_{c_2}\}\subset\{1,...,N\}.$$

\item Each ball $B(x,r)\subset\M$ contains at least one ball among $\{B(x_i,r/3)\}_{i=1}^{N}$:
$$B(x,r)\supset B(x_i,r/3)\quad\text{ for some }i\in\{1,...,N\}.$$
\end{enumerate}
Here, the positive constants $c_1$ and $c_2$ depend only on $\M$.
\end{lemma}

\begin{proof}[Proof of Lemma \ref{lemma:covering}]
Select a maximal disjoint family of balls $\{B(x_i,r/3)\}_{i=1}^{N}$ in $\M$. Then 
$$\sum_{i=1}^N\Vol(B(x_i,r/3))\le\Vol(\M).$$
So $N\le c_1r^{-n}$, where $c_1>0$ depends only on $\M$. It follows immediately from maximality that the doubled balls $B(x_i,2r/3)$, $i=1,...,N$, cover $\M$. Therefore, the family $\{B(x_i,r)\}_{i=1}^{N}$ is also a cover of $\M$.

Given any $x\in\M$, suppose that 
$$B(x,r)\subset\bigcup_{i=1}^{k}B(x_i,r)\quad\text{and}\quad B(x,r)\cap B(x_i,r)\ne\emptyset\quad\text{for all }i=1,...,k.$$
Since the disjoint balls $B(x_i,r/3)$, $i=1,...,k$, are all contained in $B(x,3r)$,
$$\sum_{i=1}^k\Vol(B(x_i,r/3))\le\Vol(B(x,3r));$$
but for each $i=1,...,k$, $B(x,3r)\subset B(x_i,5r)$ so
$$\Vol(B(x,3r))\le\Vol(B(x_i,5r))\le C_\M\Vol(B(x_i,r/3)).$$
Combing the above two inequalities, we have that $k\le c_2$, where $c_2$ depends only on $\M$. This shows (1) in the lemma.

Given any $x\in\M$, suppose that 
$$x\in B(x_i,2r/3)\quad\text{for some }i=1,...,N.$$
Then 
$$B(x_i,r/3)\subset B(x,r),$$
which proves (2) in the lemma.
\end{proof}

\noindent\textbf{Step 2}. For $0\le\alpha<1/(2n)$, take $r=\delta(h)=|\log h|^{-\alpha}$ in Lemma \ref{lemma:covering}; let $\{x_i\}_{i=1}^{N(h)}$ be the centers of the balls selected in Step 1. So $N(h)\le c_1r^{-n}=c_1|\log h|^{\alpha n}$. From Theorem \ref{thm:deltaQEinM}, 
$$V_2\big(h,a^b_{x_i}\big)=h^{n-1}\sum_{E_j\in[1,1+h]}\Big|\big\langle\Op_h(a^b_{x_i})u_j,u_j\big\rangle-\mu_1\big(a^b_{x_i}\big)\Big|^2=C\,\delta(h)^{2n}\,O(|\log h|^{-\beta})$$
as $h\to0$. Here, $C$ depends on the base function $b$, $\phi$, and $\M$, and is uniform for $x_i$. 

For $\tilde\beta>0$, write 
$$\Lambda^b_{x_i}(h)=\left\{j:E_j\in[1,1+h]\text{ and }\Big|\big\langle\Op_h(a^b_{x_i})u_j,u_j\big\rangle-\mu_1\big(a^b_{x_i}\big)\Big|\ge\delta(h)^n|\log h|^{-\tilde\beta}\right\}.$$
Hence,
$$h^{n-1}\#\Lambda^b_{x_i}(h)\le C|\log h|^{2\tilde\beta-\beta}.$$
Let 
$$\Lambda^b(h)=\bigcup_{i=1}^{N(h)}\Lambda^b_{x_i}(h).$$
Then
$$h^{n-1}\#\Lambda^b(h)\le N(h)C|\log h|^{2\tilde\beta-\beta}\le C|\log h|^{\alpha n+2\tilde\beta-\beta}.$$
Write
$$\Gamma^b(h)=\{j:E_j\in[1,1+h]\}\setminus\Lambda^b(h).$$ 
Thus, if $\alpha n+2\tilde\beta-\beta<0$, then
$$\frac{\#\Gamma^b(h)}{\#\{j:E_j\in[1,1+h]\}}=1-\frac{\#\Lambda^b(h)}{\#\{j:E_j\in[1,1+h]\}}\le1-C|\log h|^{\alpha n+2\tilde\beta-\beta}\to1$$
as $h\to0$ by Weyl's law.  Note that if $\beta-\alpha n>0$, then there exists $\tilde\beta>0$ such that
$$\alpha n+2\tilde\beta-\beta<0.$$
But this just means that $\alpha<1/(3n)$ since $0\le\beta<1-2\alpha n$. Note the existence of $\tilde\beta$ here requires $\alpha$ to be in a small range $[0,1/(3n))$ than Corollary \ref{cor:deltaequidatx}.

Now $j\in\Gamma^b(h)$ implies that
$$\Big|\big\langle\Op_h(a^b_{x_i})u_j,u_j\big\rangle-\mu_1\big(a^b_{x_i}\big)\Big|<\delta(h)^n|\log h|^{-\tilde\beta}$$
for some $0<\tilde\beta<(\beta-\alpha n)/2$, and is uniform for all $x_i$, $i=1,...,N(h)$. 

\noindent\textbf{Step 3}. In a small patch of $\M$ that contains $x_i$, similar to Corollary \ref{cor:deltaequidatx}, we use the local coordinate system $f_i=\Exp_{x_i}|_{D(0,2)}$. We still use $\tilde x$ as the the local coordinate of $x$ when $x\in B(x_i,2)$. So $\tilde x_i=0$.

Consider the symbol
$$a^b_{x_i}(x,\xi;h)=b\left(\frac{\tilde x}{\delta(h)}\right)\phi\left(|\tilde\xi|_x\right),$$
where $b\in C^\infty_0(D(0,2))$. We have
$$\big\langle\Op_h(a^b_{x_i})u_j,u_j\big\rangle=\int_{D(0,2)}b\left(\frac{\tilde x}{\delta(h)}\right)|u_j\circ f_i(\tilde x)|^2|Df_i(\tilde x)|\,d\tilde x+O(h),$$
and
$$\mu_1\left(a^b_{x_0}\right)=\int_{S^*B(0,2)}a^b_{x_0}(x,\xi;h)\,d\mu_1=\frac{1}{\Vol(\M)}\int_{D(0,2)}b\left(\frac{\tilde x}{\delta(h)}\right)|Df_i(\tilde x)|\,d\tilde x.$$
Hence, from Step 2, $j\in\Gamma^b(h)\subset\{j:E_j\in[1,1+h]\}$ implies that 
\begin{equation}\label{eq:scconv}
\left|\int_{D(0,2)}b\left(\frac{\tilde x}{\delta(h)}\right)|u_j\circ f_i(\tilde x)|^2|Df_i(\tilde x)|\,d\tilde x-\frac{1}{\Vol(\M)}\int_{D(0,2)}b\left(\frac{\tilde x}{\delta(h)}\right)|Df_i(\tilde x)|\,d\tilde x\right|<\delta(h)^n|\log h|^{-\tilde\beta}
\end{equation}
for some $\tilde\beta>0$, and is uniform for $x_i$, $i=1,...,{N(h)}$.

\noindent$\bullet$ Choose $b_1\in C^\infty_0(D(0,2))$, $0\le b_1\le1$, and $b_1=1$ in $D(0,1)$. Then \eqref{eq:scconv} implies that there is a positive constant $h_1\ll1$ such that if $j\in\Gamma^{b_1}(h)$ and $0<h\le h_1$, then
\begin{eqnarray*}
&&\int_{B(x_i,\delta(h))}|u_j(x)|^2\,d\Vol\\
&\le&\int_{D(0,2)}b_1\left(\frac{\tilde x}{\delta(h)}\right)|u_j\circ f_i(\tilde x)|^2|Df_i(\tilde x)|\,d\tilde x\\
&\le&\frac{1}{\Vol(\M)}\int_{D(0,2)}b\left(\frac{\tilde x}{\delta(h)}\right)|Df_i(\tilde x)|\,d\tilde x+\delta(h)^n|\log h|^{-\tilde\beta}\\
&\le&\frac{\Vol(B(x_i,2\delta(h)))}{\Vol(\M)}+\delta(h)^n|\log h|^{-\tilde\beta}\\
&\le&C\Vol(B(x_i,\delta(h))),
\end{eqnarray*}
which is uniform for $x_i$, $i=1,...,N(h)$. Here, $C$ depends on $b_1$, $\phi$, and $\M$.  From Step 1, each $B(x,\delta(h))$ is covered by at most $c_2$ balls among $\{B(x_i,\delta(h))\}_{i=1}^{N(h)}$, where $c_2$ only depends on $\M$. Hence, for each $B(x,\delta(h))\subset\M$,
$$\int_{B(x,\delta(h))}|u_j(x)|^2\,d\Vol\le C\Vol(B(x,\delta(h)))$$
if $j\in\Gamma^{b_1}(h)$ and $0<h\le h_1$, and is uniform for $x\in\M$. 

\noindent$\bullet$ Choose $b_2\in C^\infty_0(D(0,1/3))$, $0\le b_2\le1$, and $b_2=1$ in $D(0,1/6)$. Then \eqref{eq:scconv} implies that there is a positive constant $h_2\ll1$ such that if $j\in\Gamma^{b_2}(h)$ and $0<h\le h_2$, then
\begin{eqnarray*}
&&\int_{B(x_i,\delta(h)/3)}|u_j(x)|^2\,d\Vol\\
&\ge&\int_{D(0,1/3)}b_1\left(\frac{\tilde x}{\delta(h)}\right)|u_j\circ f_i(\tilde x)|^2|Df_i(\tilde x)|\,d\tilde x\\
&\ge&\frac{1}{\Vol(\M)}\int_{D(0,1/3)}b\left(\frac{\tilde x}{\delta(h)}\right)|Df_i(\tilde x)|\,d\tilde x-\delta(h)^n|\log h|^{-\tilde\beta}\\
&\ge&\frac{\Vol(B(x_i,\delta(h)/6))}{\Vol(\M)}-\delta(h)^n|\log h|^{-\tilde\beta}\\
&\ge&c\Vol(B(x_i,\delta(h))),
\end{eqnarray*}
which is uniform for $x_i$, $i=1,...,N(h)$. Here, $c$ depends on $b_2$, $\phi$, and $\M$. From Step 1, each $B(x,\delta(h))$ contains at least one ball among $\{B(x_i,\delta(h)/3)\}_{i=1}^{N(h)}$. Hence, for each $B(x,\delta(h))\subset\M$,
$$\int_{B(x,\delta(h))}|u_j(x)|^2\,d\Vol\ge c\Vol(B(x,\delta(h))).$$
if $j\in\Gamma^{b_2}(h)$ and $0<h\le h_2$, and is uniform for $x\in\M$. 

Let $h_0=\min\{h_1,h_2\}$ and $\Gamma(h)=\Gamma^{b_1}(h)\cap\Gamma^{b_2}(h)$. We have that if $0<h\le h_0$ and $j\in\Gamma(h)$, then 
\begin{equation}\label{eq:scequi}
c\Vol(B(x,\delta(h)))\le\int_{B(x,\delta(h))}|u_j(x)|^2\,d\Vol\le C\Vol(B(x,\delta(h)))
\end{equation}
uniformly for $x\in\M$, where the constants $c$ and $C$ depend only on $\M$. Moreover, 
\begin{equation}\label{eq:scdensity}
\frac{\#\Gamma(h)}{\#\{j:E_j\in[1,1+h]\}}\ge1-\frac{\#\Lambda^{b_1}(h)+\#\Lambda^{b_2}(h)}{\#\{j:E_j\in[1,1+h]\}}\to1\quad\text{as }h\to0.
\end{equation}

\noindent\textbf{Step 4}. We now pass from semiclassical version in Step 3 to its high-energy version. Recall the notations that $h=\lambda^{-1}$, $h^{-2}E_j=\lambda^2_j$, and $\delta(h)=r(\lambda)=(\log\lambda)^{-\alpha}$. Then 
$$\Delta u_j=h^{-2}E_ju_j=\lambda^2_ju_j$$ and 
$$\{j:E_j\in[1,1+h]\}=\{j:\lambda^2_j\in[h^{-2},h^{-2}+h^{-1}]\}.$$ 
Divide $\N$ into families
$$\begin{cases}
I_0=\left\{j:\lambda^2_j\in[0,1)\right\},\\
I_m=\left\{j:\lambda^2_j\in[h_m^{-2},h^{-2}_m+h_m^{-1})\right\},\quad\text{if }m\ge1,
\end{cases}$$
where $h_1=1$ and $h_{m+1}^{-2}=h_m^{-2}+h_m^{-1}$. Then the sequence
$$J=\bigcup_{m=0}^\infty\Gamma(h_m)$$
has full density in $\N=\cup_{m=0}^\infty I_m$. This is because
$$\frac{\#\Gamma(h)}{\#\{j:\lambda^2_j\in[h^{-2},h^{-2}+h^{-1}]\}}=\frac{\#\Gamma(h)}{\#\{j:E_j\in[1,1+h]\}}\to1\quad\text{as }h\to0$$
from \eqref{eq:scdensity}. Therefore, letting $J=\{j_k\}_{k=1}^\infty\subset\N$, in the view of \eqref{eq:scequi}, we have \eqref{eq:deltaequidinM}:
$$c\Vol(B(x,r(\lambda_{j_k})))\le\int_{B(x,r(\lambda_{j_k}))}|u_j(x)|^2\,d\Vol\le C\Vol(B(x,r(\lambda_{j_k})))\quad\text{as }k\to\infty,$$
uniformly for $x\in\M$.

\subsection{Further investigations}\hfill\\
\noindent\textbf{(1)}. Periodic and non-periodic points. The results in Theorems \ref{thm:deltaQEdeltamicrosymbol} and \ref{thm:deltaQEinM} are independent of whether the point $z_0\in S^*\M$ or $x_0\in\M$ is on a periodic trajectory or not. But the geodesic flow displays very different properties around periodic points and non-periodic points. This is reflected in Paul and Uribe \cite{PU}, where pointwise behavior of semiclassical measures is different around periodic and non-periodic points. However, one has to use more delicate information on the dynamical system $(S^*\M,G_t)$ to get qualitative results in its quantum counterpart.

\noindent\textbf{(2)}.  Different shrinking rates in $x$ and in $\xi$. Comparing the symbols that we used in Theorems \ref{thm:deltaQE} and \ref{thm:deltaQEinM}, the former one has the same shrinking rates in $x\in\M$ and in $\xi\in S^*_x\M$, while the latter only shrinks in $x\in\M$, and has a slightly better shrinking rate. So generally, one can localize better in $x$ than in $\xi$ without violating the uncertainty principle. Moreover, the pointwise version of the Weyl's law
$$\sum_{\lambda_j\le\lambda}|u_j(x)|^2=O(\lambda^n)$$
holds for all $x\in\M$. See also Christianson, Hassell and Toth \cite{CHT}, in which they used up to $h^\frac23$ localization in one direction to estimate the Neumann restriction bound of eigenfunctions. These evidence suggests that better localization in $x\in\M$ (or in $\xi$) is feasible, e.g. to consider symbols as
$$a^b_{x_0}(x,\xi;h)=b\left(\frac{x-x_0}{h^\rho}\right)\phi(|\xi|_x),\quad\rho\in(0,1).$$

\noindent\textbf{(3)}.  Manifolds with constant curvature or arithmetic hyperbolic manifolds. The main tool we used from dynamical system is the exponential rate of correlation \eqref{eq:rom} by Liverani \cite{Liv}. His result, however, is on general contact Anosov flows. In Theorem \ref{thm:roe}, the rate of ergodicity is proved to be controlled by the H\"older norm. Its improvement may lead to better localization in Theorems \ref{thm:deltaQE} and \ref{thm:deltaQEinM}. One expects that this can be done in manifolds with constant curvature or arithmetic hyperbolic manifolds, in which cases there are group symmetry or number theoretic results to exploit.

\section*{Acknowledgements}
I would like to thank Gabriel Rivi\`ere for informing me Liverani's paper \cite{Liv} and pointing out that Theorem \ref{thm:roe} follows from it. I also want to thank Hamid Hezari and St\'ephane Nonnenmacher for their comments that helped to improve the paper. It is a pleasure to thank Andrew Hassell for encouraging me to work on small scale quantum ergodicity and all our discussions throughout the preparation of this paper. Finally, I would like to thank the  referees for very careful reading and detailed comments that helped to improve the presentation.


\begin{thebibliography}{99}

\bibitem[Ana]{Ana} N. Anantharaman,
\textit{Entropy and the localization of eigenfunctions}. Ann. of Math. (2) 168 (2008), no. 2, 435--475. 

\bibitem[Ano]{Ano} D. Anosov, 
\textit{Geodesic flows on closed Riemannian manifolds of negative curvature}. Trudy Mat. Inst. Steklov. 90 1967.

\bibitem[AN]{AN} N. Anantharaman and S. Nonnenmacher, 
\textit{Half-delocalization of eigenfunctions for the Laplacian on an Anosov manifold}. Festival Yves Colin de Verdi\`ere. Ann. Inst. Fourier (Grenoble) 57 (2007), no. 7, 2465--2523.

\bibitem[B]{B} W. Ballmann, 
\textit{Lectures on spaces of nonpositive curvature}. With an appendix by Misha Brin. Birkh\"auser Verlag, Basel, 1995.

\bibitem[BL]{BL} S. Brooks and E. Lindenstrauss, 
\textit{Joint quasimodes, positive entropy, and quantum unique ergodicity}. Invent. Math. 198 (2014), no. 1, 219--259.

\bibitem[BR]{BR} A. Bouzouina and D. Robert,
\textit{Uniform semiclassical estimates for the propagation of quantum observables}. Duke Math. J. 111 (2002), no. 2, 223--252.

\bibitem[CdV]{CdV} Y. Colin de Verdi\`ere,
\textit{Ergodicit\'e et fonctions propres du laplacien}. Comm. Math. Phys. 102 (1985), no. 3, 497--502.

\bibitem[CHT]{CHT} H. Christianson, A. Hassell, and J. Toth, 
\textit{Exterior mass estimates and $L^2$ restriction boundary for Neumann data along hypersurfaces}. \href{http://arxiv.org/abs/1303.4319}{arXiv:1303.4319}. To appear in Int. Math. Res. Not. IMRN.

\bibitem[CM]{CM} T. Colding and W. Minicozzi,
\textit{Lower bounds for nodal sets of eigenfunctions}. Comm. Math. Phys. 306 (2011), no. 3, 777--784.

\bibitem[DF]{DF} H. Donnelly and C. Fefferman,
\textit{Nodal sets of eigenfunctions on Riemannian manifolds}. Invent. Math. 93 (1988), no. 1, 161--183.

\bibitem[DuGu]{DuGu} J. J. Duistermaat and V. W. Guillemin, 
\textit{The spectrum of positive elliptic operators and periodic bicharacteristics}. Invent. Math. 29 (1975), 39--79.

\bibitem[DyGu]{DyGu} S. Dyatlov and C. Guillarmou,
\textit{Microlocal limits of plane waves and Eisenstein functions}. Ann. Sci. \'Ec. Norm. Sup\'er. (4) 47 (2014), no. 2, 371--448.

\bibitem[Ha]{Ha} A. Hassell,
\textit{Ergodic billiards that are not quantum unique ergodic}. With an appendix by the author and Luc Hillairet. Ann. of Math. (2) 171 (2010), no. 1, 605--619. 

\bibitem[Ho]{Ho} L. H\"ormander,
\textit{The spectral function of an elliptic operator}. Acta Math. 121 (1968), 193--218.

\bibitem[HMR]{HMR} B. Helffer, A. Martinez, and D. Robert, 
\textit{Ergodicit\'e et limite semi-classique}. Comm. Math. Phys. 109 (1987), no. 2, 313--326.

\bibitem[HR1]{HR1} H. Hezari and G. Rivi\`ere,
\textit{$L^p$ norms, nodal sets, and quantum ergodicity}. \href{http://arxiv.org/abs/1411.4078}{arXiv:1411.4078}.

\bibitem[HR2]{HR2} H. Hezari and G. Rivi\`ere,
\textit{Quantitative equidistribution properties of toral eigenfunctions}. \href{http://arxiv.org/abs/1503.02794}{arXiv:1503.02794}.

\bibitem[HS]{HS} R. Holowinsky and K. Soundararajan,
\textit{Mass equidistribution for Hecke eigenforms}. Ann. of Math. (2) 172 (2010), no. 2, 1517--1528.

\bibitem[KH]{KH} A. Katok and B. Hasselblatt,
\textit{Introduction to the modern theory of dynamical systems}. Cambridge University Press, Cambridge, 1995.

\bibitem[Lin]{Lin} E. Lindenstrauss, 
\textit{Invariant measures and arithmetic quantum unique ergodicity}, Ann. of Math. (2) 163 (2006), no. 1, 165--219.

\bibitem[Liv]{Liv} C. Liverani,
\textit{On contact Anosov flows}. Ann. of Math. (2) 159 (2004), no. 3, 1275--1312. 

\bibitem[LS]{LS} P. Li and R. Schoen,
\textit{$L^p$ and mean value properties of subharmonic functions on Riemannian manifolds}. Acta Math. 153 (1984), no. 3-4, 279--301.

\bibitem[PU]{PU} T. Paul and T. Uribe,
\textit{On the pointwise behavior of semi-classical measures}. Comm. Math. Phys. 175 (1996), no. 2, 229--258.

\bibitem[R]{R} M. Ratner, 
\textit{The central limit theorem for geodesic flows on $n$-dimensional manifolds of negative curvature}. Israel J. Math. 16 (1973), 181--197.

\bibitem[RS]{RS} Z. Rudnick and P. Sarnak, 
\textit{The behaviour of eigenstates of arithmetic hyperbolic manifolds}. Comm. Math. Phys. 161 (1994), no. 1, 195--213.

\bibitem[Sa1]{Sa1} P. Sarnak,
\textit{Arithmetic quantum chaos}. The Schur lectures (1992) (Tel Aviv), 183--236, Israel Math. Conf. Proc., 8, Bar-Ilan Univ., Ramat Gan, 1995.

\bibitem[Sa2]{Sa2} P. Sarnak,
\textit{Recent progress on the quantum unique ergodicity conjecture}. Bull. Amer. Math. Soc. (N.S.) 48 (2011), no. 2, 211--228.

\bibitem[Sc]{Sc} R. Schubert, 
\textit{Upper bounds on the rate of quantum ergodicity}. Ann. Henri Poincar\`e 7 (2006), no. 6, 1085--1098.

\bibitem[So1]{So1} C. Sogge,
\textit{Hangzhou lectures on eigenfunctions of the Laplacian}. Princeton University Press, Princeton, NJ, 2014.

\bibitem[So2]{So2} C. Sogge,
\textit{Localized $L^p$-estimates of eigenfunctions: A note on an article of Hezari and Rivi\`ere}. \href{http://arxiv.org/abs/1503.07238}{arXiv:1503.07238}.

\bibitem[Sn]{Sn} A. \v Snirel'man,
\textit{The asymptotic multiplicity of the spectrum of the Laplace operator}. (Russian) Uspehi Mat. Nauk 30 (1975), no. 4 (184), 265--266.

\bibitem[SV]{SV} L. Silberman and A. Venkatesh, 
\textit{On quantum unique ergodicity for locally symmetric spaces}, Geom. Funct. Anal. 17 (2007), 960--998.

\bibitem[Ze1]{Ze1} S. Zelditch,
\textit{Uniform distribution of eigenfunctions on compact hyperbolic surfaces}. Duke Math. J. 55 (1987), no. 4, 919--941.

\bibitem[Ze2]{Ze2} S. Zelditch,
\textit{On the rate of quantum ergodicity. I. Upper bounds}. Comm. Math. Phys. 160 (1994), no. 1, 81--92.

\bibitem[Ze3]{Ze3} S. Zelditch,
\textit{Recent developments in mathematical quantum chaos}. Current developments in mathematics, 2009, 115--204, Int. Press, Somerville, MA, 2010.

\bibitem[Zw]{Zw} M. Zworski,
\textit{Semiclassical analysis}. American Mathematical Society, Providence, 2012.

\end{thebibliography}
\end{document}